\documentclass[10pt]{amsart}
\usepackage{amssymb}
\usepackage{epsfig}
\usepackage{url}
\usepackage{setspace}
\usepackage{pdflscape}
\theoremstyle{plain}

\newtheorem{thm}{Theorem}[section]
\newtheorem{cor}[thm]{Corollary}

\newtheorem{rem}[thm]{Remark}

\newtheorem{conj}[thm]{Conjecture}

\newtheorem{prob}[thm]{Problem}
\def\cal{\mathcal}
\def\bbb{\mathbb}
\def\op{\operatorname}
\renewcommand{\phi}{\varphi}

\newcommand{\N}{\bbb{N}}
\newcommand{\Z}{\bbb{Z}}
\newcommand{\Q}{\bbb{Q}}

\begin{document}

\title[Primitive integer solutions of certain Diophantine equations]{On primitive integer solutions of the Diophantine equation $t^2=G(x,y,z)$ and related
results} \author{Maciej Gawron and Maciej Ulas}

\thanks{The research of the authors was supported by Polish National Science Centre grants:  UMO-2014/13/N/ST1/02471 (MG) and UMO-2012/07/E/ST1/00185 (MU)}

%MG was supported by . The research of the second author was partially supported by the grant of the Polish National Science %Centre no. UMO-2012/07/E/ST1/00185}

\begin{abstract}
In this paper we investaigate Diophantine equations of the form $T^2=G(\overline{X}),\; \overline{X}=(X_{1},\ldots,X_{m})$, where mainly $m=3$ or $m=4$ and $G$ specific homogenous quintic form.
%We show that for a broad class of quadratic forms $F$ in three variables, the Diophantine equation $t^2=nxyzF(x,y,z)$,
%where $n\in\Z\setminus\{0\}$, has a polynomial solution in $\Z[u,v]$ without constant common factor.
First, we prove that if $F(x,y,z)=x^2+y^2+az^2+bxy+cyz+dxz\in\Z[x,y,z]$ and $(b-2,4a-d^2,d)\neq (0,0,0)$, then the Diophantine equation $t^2=nxyzF(x,y,z)$ has solution in polynomials $x, y, z, t$ with integer coefficients, without polynomial common factor of positive degree. In case $a=d=0, b=2$ we prove that there are infinitely many primitive integer solutions of the Diophantine equation under consideration. As an application of our result we prove that for each $n\in\Q\setminus\{0\}$ the Diophantine equation
\begin{equation*}
T^2=n(X_{1}^5+X_{2}^5+X_{3}^5+X_{4}^5)
\end{equation*}
has a solution in co-prime (non-homogenous) polynomials in two variables with integer coefficients. We also present a method which sometimes  allow us to prove the existence of primitive integers solutions of more general quintic Diophantine equations of the form $T^2=aX_{1}^5+bX_{2}^5+cX_{3}^5+dX_{4}^5$, where $a, b, c, d\in\Z$. In particular, we prove that for each $m, n\in\Z\setminus\{0\}$ , the Diophantine equation
\begin{equation*}
T^2=m(X_{1}^5-X_{2}^5)+n^2(X_{3}^5-X_{4}^5)
\end{equation*}
has a solution in polynomials which are co-prime over $\Z[t]$. Moreover, we show how modification of the presented method can be used in order to prove that for each
$n\in\Q\setminus\{0\}$, the Diophantine equation
\begin{equation*}
t^2=n(X_{1}^5+X_{2}^5-2X_{3}^5)
\end{equation*}
has a solution in polynomials which are co-prime over $\Z[t]$.

%This implies that each integer $n$ can be written in infinitely many ways as a quotient of a square and a sum of four fifth powers.

\end{abstract}

\maketitle

\section{Introduction}\label{sec1}

The aim of this note is to present some results concerning the existence of primitive solutions of certain non-homogenous Diophantine equations of the form
\begin{equation}\label{general} t^2=G(x,y,z), \end{equation} where $G\in \Z[x,y,z]$ is homogenous form of degree 5. The above equation is a very special case of
the Diophantine equation
$$F(\overline{X}_{m})=G(\overline{Y}_{n}),$$
 where $F\in\Z[\overline{X}_{m}], G\in\Z[\overline{Y}_{n}]$ are homogenous forms with
co-prime degrees, and $\overline{T}_{k}=(T_{1},\ldots,T_{k})$. We observe that the question concerning the existence of integer solutions of this equation is
easy. Indeed, if $\op{deg}F=a, \op{deg}G=b$ with $\gcd(a,b)=1$ then there are positive integer $\alpha, \beta$ such that $\alpha a-\beta b=1$. Thus, the
$m$-tuple $\overline{X}_{m}$ with  $X_{i}=x_{i}T^{\alpha}, i=1,\ldots, m$ and $n$-tuple $\overline{Y}_{n}$ with $Y_{j}=y_{j}F(\overline{x})T^{\beta},
j=1,\ldots,n$, where $T=F(\overline{x})^{b-1}F(\overline{y})$ satisfy the equation $F(\overline{X})=G(\overline{Y})$. However, in general the solutions obtained in this way have common factor and it is an interesting and non-trivial question whether we can find polynomial solutions without constant common factors or solutions in co-prime integers.

The direct motivation to write this paper was a question whether there are infinitely many quadruplets of integers with such a property the sum of its fifth
powers is a square. This question can be seen as a very special case of a more general problem which is called {\it equal sums of unlike powers}. This problem
was investigated by Lander in \cite{Lan}. In the cited paper Lander presented some methods which allow to find integer solutions of the related Diophantine
equation. In this contents we also should mention a paper by Schinzel \cite{Sch1} (see also \cite{Sch2}), where more general Diophantine equations involving
powers are considered. Unfortunately, in most cases these methods produce solutions which are not co-prime. In particular, it is not know whether there are
infinitely many quadruplets of integers $(X_{1},X_{2},X_{3},X_{4})$ with sum of its fifth powers equal a square. We relate this question to the question of
existence of primitive solutions of the equation (\ref{general}) for specific choice of the homogenous form $G$.

Let us describe the content of the paper in some details. In section \ref{sec2} we consider the Diophantine equation $t^2=nxyzF(x,y,z)$ with
$F(x,y,z)=x^2+y^2+az^2+bxy+cyz+dxz$ and $n\in\Z\setminus\{0\}$. We prove that under assumption $(b-2,4a-d^2,d)\neq (0,0,0)$ the equation under consideration has a solution in polynomials $x, y, z\in\Z[u,v]$ such that $\gcd_{\Q[u,v]}(x,y,z)=1$. In the case $a=d=0, b=2$ we prove that there are infinitely many primitive integer solutions. In the next section we present an amusing application of this result and prove that for each $n\in\Q\setminus\{0\}$ the Diophantine equation
$$T^2=n(X_{1}^5+X_{2}^5+X_{3}^5+X_{4}^5)$$ has solution in co-prime polynomials with integer coefficients (Theorem \ref{fifth}). This application is possible due to the identity
\begin{equation*}
F_{2}=80xyz(x^2+y^2+z^2),
\end{equation*}
where
\begin{equation*}
F_{n}=(x+y+z)^{2n+1}+(-x+z-y)^{2n+1}+(x-y-z)^{2n+1}+(-x+y-z)^{2n+1}.
\end{equation*}
We make a rather long digression of independent interest and present an interesting arithmetic property of the sequence $\{F_{n}\}_{n\in\N_{+}}$.
More precisely, we find an explicite value of the highest power of an prime number $p$ which divide $F_{n}$ (as a polynomial in $\Z[x,y,z]$). Finally, in the
last section we present a method which sometimes allows to find polynomial solutions without common factor of positive degree of the Diophantine equation $T^2=aX_{1}^5+bX_{2}^5+cX_{3}^5+dX_{4}^5$. In particular, we prove that if $m, n\in\Z\setminus\{0\}$ then one can find co-prime polynomials $X_{1}, X_{2}, X_{3}, X_{4}\in\Z[t]$ of degree 4 and a polynomial $T$ such that
$$T^2=m(X_{1}^5-X_{2}^5)+n^2(X_{3}^5-X_{4}^5).
$$
This result is given as Theorem \ref{fifth-general}. Moreover, we also prove that for each non-zero rational number $n$  the Diophantine equation
$$
t^2=n(x^5+y^5-2z^5).
$$
has solution in co-prime polynomials with integer coefficients. At the end we discuss some numerical results and state certain conjectures which may stimulate further research.

\section{The equation $t^2=nxyzF(x,y,z)$}\label{sec2}

In this section we present an approach which allow us to find polynomial solutions without non-constant common factors of the Diophantine equation $t^2=nxyzF(x,y,z)$, where $n\in\Z\setminus\{0\}$
and
\begin{equation}\label{defF}
F(x,y,z)=x^2+y^2+az^2+bxy+cyz+dxz.
\end{equation}
More precisely, we prove the following:

\begin{thm}\label{F-thm} Let $n\in\Z\setminus\{0\}$ be given and let $F$ be a quadratic form given by {\rm (\ref{defF})} with $a,b,c,d\in\Z$. Let us consider the
diophantine equation
\begin{equation}\label{diop}
 t^2=nxyzF(x,y,z).
\end{equation} If $(b-2,4a-d^2,d)\neq (0,0,0)$ then the equation {\rm (\ref{diop})} has (non-homogenous) polynomial solution $x,y,z,t\in\Z[u,v]$ satisfying the
condition $\gcd_{\Q[u,v]}(x,y,z)=1$, moreover $\gcd_{\Z[u,v]}(x,y,z)|2n$.

If $a=d=0, b=2$ then the equation {\rm (\ref{diop})} has infinitely many primitive integer solutions.
\end{thm}
\begin{proof}
We proceed as follows. We put
\begin{equation}\label{sub1}
x=npq,\quad y=qr,\quad z=pr, \quad t=npqrw
\end{equation} and get the equality
\begin{equation*}
t^2-nxyzF(x,y,z)=n^2p^2q^2r^2(w^2-(n^2p^2q^2+n(dp+bq)pqr + (ap^2+cpq+q^2) r^2)).
\end{equation*} We thus left with the problem of finding integral solutions of
the diophantine equation
\begin{equation}\label{eq1}
w^2=n^2p^2q^2+n(dp+bq)pqr + (ap^2+cpq+q^2)r^2.
\end{equation}
Although we are interested in integral solutions we will concentrate on finding {\it rational solutions}. This is an easy task. Indeed, we can view the equation (\ref{eq1}) as an equation of quadratic
curve, say $\cal{C}$, defined over the rational function field $\Q(p,q)$ in the plane $(r,w)$. Nothing that $\cal{C}$ contains the $\Q(p,q)$-rational point
$(r,w)=(0,npq)$ we can find parametrization of solutions in the following form
\begin{equation}\label{rwsol}
r=\frac{npq(2s-dp-bq)}{ap^2+cpq+q^2-s^2},\quad
w=\frac{npq(ap^2+cpq+q^2-(dp+bq)s+s^2)}{ap^2+cpq+q^2-s^2},
\end{equation}
where $s$ is a rational parameter. Now, the crucial question arises: it is possible to
find a rational substitution for $p, q, s$ with such a property that each $x=npq, y=qr, z=pr$ with $r$ given above, take the form $PS^2, QS^2, RS^2$, where $P,
Q, R$ are polynomials without common factor and $S$ is a rational function? If yes, then necessarily the expression for $t$ is of the form $WS^5$ for some
polynomial $W$ and the quadruple $(x, y, z, t)=(P, Q, R, W)$, satisfying the condition $\gcd(x,y,z)=1$, is a solution of the equation $t^2=nxyzF(x,y,z)$. In
order to find a substitution we are looking for we define the rational functions $p$ and $s$ as the solution of the equation
\begin{equation*}
npq=(ap^2+cpq+q^2-s^2)v^2,
\end{equation*}
where $v$ is a rational parameter. It is easy to see that the considered equation defines a rational curve, say
$\cal{C}'$, defined over $\Q(q,v)$ in the plane $(p,s)$. The curve $\cal{C}'$ contains the point $(p,s)=(0,q)$ and thus by putting $s=q+pu$ we immediately get
the parametrization of $\cal{C}'$ in the following form
\begin{equation*}
p=\frac{q(n+(2u-c)v^2)}{v^2(a-u^2)},\quad s=\frac{q(nu+(a-cu+u^2)v^2)}{v^2(a-u^2)},
\end{equation*}
where $u$ is a rational parameter. We observe that for each expression  $U\in\{x, y, z\}$, where $x, y, z$ are given by (\ref{sub1}) with $p, s$
given above and corresponding $r, w$ given by (\ref{rwsol}) we have $U=q^2\frac{V_{1}}{V_{2}}$. Moreover, $V_{1}, V_{2}\in\Z[u,v]$ with $V_{2}$ dividing the
polynomial $(a - u^2)^2 v^2$. We thus put $q=(a-u^2)v$ and get (after cancelation of the square common factors from $x, y, z$) the solution of the equation
$t^2=nxyzF(x,y,z)$ in the following form
\begin{align}\label{pos1}
x&=n(a-u^2)(n+(2u-c)v^2), \nonumber\\
y&=(a-u^2)v^2(n(2u-d)+(2a-ab+cd-2(c+d)u + (b+2)u^2)v^2),\\
z&=(n+(2u-c)v^2)(n(2u-d)+(a(2-b)+cd-2(c+d)u + (b+2)u^2)v^2), \nonumber
\end{align} with the (polynomial) expression for $t$ given by
\begin{equation*}
t=\frac{xy}{(a-u^2)nv}(n+(2-b)u v^2)xy-nu(d-2u)(n-(c-2u)v^2)^2).
\end{equation*}

Let us observe that if we treat $x, y, z$ given above as polynomials defined over $K:=\Q[a,b,c,d]$ then we have $\gcd_{K[u,v]}(x,y,z)=1$. However, we are
essentially interested in integer specializations of $a, b, c, d$. In order to finish the proof we need to characterize the set of those
quadruplets $(a,b,c,d)$ of integers such that $\gcd(x,y,z)=1$ in $\Q[u,v]$. We thus fix $a,b,c,d\in \Z$. We consider two cases $b\neq 2$ and $b=2$.

Suppose that $b\neq 2$. First of all let us note that $a-u^2\mid \gcd(x,y)$. Moreover, we have
\begin{equation*}
\op{Res}_{u}(a-u^2,z)=(4a-d^2)(n^2-2cnv^2-(4a-c^2)v^4)^2
\end{equation*}
 and thus $a-u^2$ has a (polynomial) common factor with $z(u,v)$ if and only if $4a=d^2$. We thus put
$a=\frac{1}{4}d^2$. Under this assumption we get that $x, y, z$ has a common factor $d-2u$. However, this factor can be easily eliminated by introducing new
variable $w$ chosen in such a way that $2u-d=w^2$. Thus we put $u:=U(w)=\frac{1}{2}(d+w^2)$. Then we observe that the rational functions $x(U,v), y(U,v), z(U,v)$
has common factor $(w/4)^2$. Eliminating this factor from $x(U,v), y(U,v), z(U,v)$ we get the solution of the equation $t^2=nxyzF(x,y,z)$ in the following form:
\begin{align}\label{pos2}
x_{1}&=4n(2 d + w^2)((c - d - w^2)v^2-n),\nonumber\\
y_{1}&=v^2w^2(2d+w^2)(2(2c-bd)v^2-(2+b)v^2w^2-4n),\\
z_{1}&=4((c - d-w^2)v^2-n)((2+b)v^2w^2-2(2c-bd)v^2+4n)\nonumber.
\end{align}
%(which are elements of the ring $\Z[w,v]$) have common factor $w^2/16$ and thus we can divide both sides of
%the identity $t^2=nxyzF(x,y,z)$ by $w^{10}/2^{20}$.

%We thus put %\begin{equation} %x_{2}(w,v)=16x_{1}(w,v)/w^2,\quad  y_{2}(w,v)=16y_{1}(w,v)/w^2, \quad z_{2}(w,v)=16z_{1}(w,v)/w^2 %\end{equation}

However, it is still possible that the polynomials $x_{1}, y_{1}, z_{1}$ may have common polynomial factor for certain choices of $a, c, d$. We show that this is
not the case under assumption from the statement of our theorem. First of all, from our assumption on $n$, we note that the polynomial $v^2w^2$ has no common
factors with the polynomial $x_{1}z_{1}$. Next, we observe that $2d+w^2\mid \gcd(x_{1}, y_{1})$ and
\begin{equation*}
\op{Res}_{w}(2d+w^2,z_{1}(w,v))=(2d+w^2)^4.
\end{equation*}
This identity implies that the polynomial $2d+w^2$ has no common factor with $z_{1}$. Similarly, we observe that $n+(w^2+d-c)v^2|\gcd(x_{1},
z_{1})$ and we easily get
\begin{equation*}
\frac{\op{Res}_{v}(n+(w^2+d-c)v^2,z_{1})}{n^4w^4(2d+w^2)^4}=\frac{\op{Res}_{w}(n+(w^2+d-c)v^2,z_{1})}{v^8((c-d)v^2-n)^2 ((c+d)v^2-n)^4}=(b-2)^2.
\end{equation*}
It is clear that the denominators of the expressions above are non-zero for all choices of $a, d, c\in\Z$ and $n\in\Z\setminus\{0\}$. Thus the necessary
condition to the existence of common factor of $x_{1}, y_{1}, z_{1}$  is $b=2$ - the case we excluded.

We consider the case $b=2$ and suppose that $4a\neq d^2$. If $b=2$ then we have $\gcd(x,y,z)=n+(2u-c)v^2$ and thus in order to eliminate this factor we need to
make it square. In order to do this we put $u:=U(v,w)=(w^2+cv^2-n)/(2 v^2)$ and then observe that the rational functions $x(U,v), y(U,v), z(U,v)$ has common
factor $(2v^2/w)^2$. Eliminating this factor from $x(U,v), y(U,v), z(U,v)$ we get the solution of the equation $t^2=nxyzF(x,y,z)$ in the following form:
\begin{align}\label{pos3}
x_{2}&=-n(w^4-2 (n - c v^2)w^2+n^2-2cnv^2+(c^2-4a)v^4), \nonumber\\
y_{2}&=(n-w^2-(c-d)v^2)(w^4-2 (n - c v^2)w^2+n^2-2cnv^2+(c^2-4a)v^4),\\
z_{2}&=4v^2w^2(w^2+(c-d)v^2-n). \nonumber
\end{align}
Because $n\neq 0$ then the polynomials $v^2w^2, x_{2}$ has no common factors. Next, we observe that the polynomials
$x_{2}, z_{2}$ have common factor if and only if $\op{Res}_{v}(x_{2},z_{2})=0$ or $\op{Res}_{w}(x_{2},z_{2})=0$. However, we have
\begin{equation*}
\frac{\op{Res}_{v}(x_{2},z_{2})}{(n - w^2)^4}=\frac{\op{Res}_{w}(x_{2},z_{2})}{v^8}=(4 a - d^2)^2 n^2,
\end{equation*}
and thus from the non-vanishing of the
denominators of the expressions given above we deduce that the polynomials $x_{2}, z_{2}$ have common factor if and only if $4a=d^2$ - the case we excluded.

We consider the last case: $b=2, 4a=d^2$ and $d\neq 0$. In this case we have $\gcd(x,y,z)=(d-2 u)((c-2u)v^2-n)$ and thus in order to eliminate this factor we
need to make it square. It is easy to find a rational function $u$ such that $(d-2 u)((c-2u)v^2-n)$ is a square. Indeed, it is enough to take
$u:=U(v,w)=(dt^2+cv^2-n)/(2(t^2+v^2))$ and the the rational functions $x(U,v), y(U,v), z(U,v)$ has common factor $(2(v^2+w^2)^2)/((n-c v^2+d v^2)w)^2$. We thus
get the solution of the Diophantine equation $t^2=nxyzF(x,y,z)$ in the following form:
\begin{align}\label{pos4}
x_{3}&=-n(v^2+w^2)(n-(c+d)v^2-2dw^2), \nonumber\\
y_{3}&=v^2(n-(c-d)v^2)(n-(c+d)v^2-2dw^2),\\
z_{3}&=-4(n-(c-d)v^2)w^2(v^2 + w^2). \nonumber
\end{align}
Because $n\neq 0$ we have $\gcd(x_{3}y_{3}, w)=1$. Moreover, the
polynomials $x_{3}, y_{3}, z_{3}$ have common factor if and only if $\op{Res}_{v}(x_{3},n-(c-d)v^2)=0$ or $\op{Res}_{v}(y_{3},v^2+w^2)=0$. However, because
$dn\neq 0$ the expression $\op{Res}_{v}(x_{3},n-(c-d)v^2)=4d^2n^2(n+(c-d)w^2)^4$ is non-zero. Similarly, the expression
$\op{Res}_{v}(y_{3},v^2+w^2)=w^4(n+(c-d)v^2)^4$ can not be zero too.

Summing up: our discussion we see that under assumption $(b-2,4a-d^2,d)\neq (0,0,0)$ we get solution of the Diophantine equation (\ref{diop}) in polynomials with integer coefficients without common polynomial factor. One can check (by examination of the coefficients of our solutions) that in each case $\gcd_{\Z[u,v]}(x,y,z)$ divides $2n$. First part of our theorem follows.

We consider the case $a=d=0, b=2$. In this case we have
\begin{equation*}
 t^2 = nxyz ( x^2+y^2+2xy+cyz) = nxyz( (x+y)^2 +cyz)).
\end{equation*}
If $n$ is not a square we can put $x=y=cnv^2$ and $z=1$ to get
\begin{equation*}
t^2 = c^4n^4v^6(4nv^2+1).
\end{equation*}
Here $v$ is a variable. By substitution $u = t/(c^2n^2v^3)$ we get the equation \mbox{$u^2-4nv^2=1$} which is a Pell equation and by Lagrange theorem we know that it has infinitely many integer solutions.

When $n$ is a perfect square, we take $x=y=cp^2q^2$, $z = (p^2-q^2)^2$ and get
\begin{equation*}
 t^2 = nc^4p^6q^6(p^4-q^4)^2.
\end{equation*}
We can ensure that $\gcd(x,y,z) = 1$ by taking $p=cs+1$ and $q=cs$, for some $s$ and we are done. \end{proof}

\begin{cor}\label{corn1}
If $n=1$, then the equation {\rm (\ref{diop})} has a solution in polynomials $x, y, z, t\in\Z[u,v]$ with $\op{gcd}_{\Z[u,v]}(x,y,z)=1$.
\end{cor}
\begin{proof}
 When the solution is given by one of the forms (\ref{pos1}), (\ref{pos3}), (\ref{pos4})  then one of the coefficients of $x(u,v)$ is equal $-1$, and thus due to coprimality of $x,y,z$ over $\Q[u,v]$ we get that these polynomials are coprime in $\Z[u,v]$.  The only nontrivial case is when the solution is given by (\ref{pos2}).
 In that case we put $v=4\hat{v}, w = \hat{w}/2$. Then we get we get polynomials $x_1, y_1, z_1\in\Z[\hat{v},\hat{w}]$ which give solution of our equation. Moreover, the coefficient at $\hat{w}^2$ in $x(\hat{v},\hat{w})$ is $-1$ and thus the polynomials $x_1, y_1, z_1$ are coprime in $\mathbb{Z}[\hat{w},\hat{v}]$.

\end{proof}

\begin{rem}
{\rm Let us note that we essentially proved that if $(b-2,4a-d^2,d)\neq (0,0,0)$ then the system of the Diophantine equations
\begin{equation*}
t_{1}^2=nxyz,\quad t_{2}^2=x^2+y^2+az^2+bxy+cyz+dxz
\end{equation*}
has solution in polynomials $x(u,v),y(u,v),z(u,v)$, such that $\gcd_{\Q[u,v]}(x,y,z)=1$.
}
\end{rem}

Although we trying quite hard to prove the existence of polynomial solutions of the equation (\ref{diop}) satisfying the condition $\gcd_{\Z[u,v]}(x,y,z)=1$ we failed. However, in the light of the above result we expect that the following is true.

\begin{conj}
 For any quadratic form $F\in\Z[x,y,z]$ the Diophantine equation~{\rm (\ref{diop})} has:
\begin{enumerate}
 \item[(a)]  a solution in polynomials $x,y,z$ with integer coefficients with $\gcd(x,y,z)=1$;
 \item[(b)]  infinitely many primitive solutions in integers.
\end{enumerate}
\end{conj}

\begin{rem}
{\rm It seems that essentially the same type of reasoning as presented in the proof of Theorem \ref{F-thm} can be used in order to find co-prime (in $\Q[u,v]$, say) polynomial solutions of the general Diophantine equation of the form
\begin{equation*}
t^2=nx_{1}x_{2}\ldots x_{k}F(x_{1},\ldots,x_{k}),
\end{equation*} where
$F(x_{1},\ldots,x_{k})=\sum_{i\leq j\leq k}a_{ij}x_{i}x_{j}$ is a quadratic form, with, say, $a_{11}=a_{22}=1$ and $a_{ij}\in\Z$. Here we can put
$x_{1}=np_{1}p_{2}, x_{i}=p_{i}p_{i+1}$ for $i=2,\ldots,k$ with the convention $p_{k+1}=p_{1}$ and proceed essentially in the same way as in the proof of Theorem \ref{F-thm} (with respect to $p_{1}$ and $p_{2}$). }
\end{rem}

\section{Parametric solution of the Diophantine equation $T^2=n(X_{1}^5+X_{2}^5+X_{3}^5+X_{4}^5)$}\label{sec3}

The aim of this section is to give an application of the result we get in section \ref{sec2}. More precisely, we present one result concerning the existence of
primitive integer solutions of the Diophantine equations involving fifth powers. However, before we state our result let us note that Mostafa in \cite{Mo}
observed that for
\begin{equation*}
x=5pq, \quad y=5(p^2+pq+q^2),\quad z=-5p(p+q),\quad t=-5q(p+q),
\end{equation*}
we have $x^5+y^5+z^5+t^5=d^2$ with
$d=125pq(p+q)(p^2+pq+q^2)$. The presented solution has two drawbacks. First of all it is not clear how this solution was found (it is clear that this is not a real
drawback). However, the second one is the more important. It is easy to see that $\gcd$ of the polynomials $x, y, z, t\in\Z[p,q]$ is equal to 5. Unfortunately,
it is clear from the shape of $x, y, z, t, d$ that is impossible to eliminate common factor 5 in order to get solution of the equation $x^5+y^5+z^5+t^5=d^2$ in
co-prime polynomials with integer coefficients and without constant common factors. This suggest a natural problem:

\begin{prob} Find solution of the Diophantine equation $T^2=X_{1}^5+X_{2}^5+X_{3}^5+X_{4}^5$ in co-prime polynomials $X_{1}, X_{2}, X_{3}, X_{4}$ and without constant common factor $>1$. \end{prob}

Remarkably with the help of the result we proved in the previous section we can obtain the following:

\begin{thm}\label{fifth}
Let $n\in\Q\setminus\{0\}$ be given. Then the Diophantine equation
\begin{equation}\label{5powers}
T^2=n(X_{1}^5+X_{2}^5+X_{3}^5+X_{4}^5)
\end{equation} has solution in co-prime polynomials $X_{1}, X_{2}, X_{3}, X_{4}, T\in\Z[u,v]$.
\end{thm}
\begin{proof} First we consider the case when $n$ is a non-zero integer. In order to get the result we put
\begin{equation}\label{Xexpression}
X_{1}=x+y+z, X_{2}=z-x-y, X_{3}=x-y-z, X_{4}=y-x-z.
\end{equation}
With $X_{i}, i=1,2,3,4$, defined in this way we have an
equality
\begin{equation*}
 n(X_{1}^5+X_{2}^5+X_{3}^5+X_{4}^5)=4^2\cdot 5nxyz(x^2+y^2+z^2).
\end{equation*}
We will apply the result obtained in Theorem \ref{F-thm}. In order to do this we replace $n$ by $5n$ and consider $F(x,y,z)=x^2+y^2+z^2$. Then,
following the arguments presented in the proof of Theorem \ref{F-thm} we get the solutions of the equation $(T/4)^2=5nxyzF(x,y,z)$ in the following form:
\begin{align*}
x&=5n(1-u)(1+u)(5n+4uv^2),\\
y&=2(1-u)(1+u)v^2(5nu + 2(1+u^2)v^2),\\
z&=2(5n+4uv^2)(5nu+2(1+u^2)v^2).
\end{align*}
Substituting the expressions
for $x, y, z$ given above into  $X_{i}, i=1,2,3,4$, given by (\ref{Xexpression}), and performing all necessary simplifications we get the solution of the
equation (\ref{5powers}) in the following form:
\begin{align*}
X_{1}&=25n^2 (1+2 u-u^2)+10 n(1+2 u+3 u^2-2 u^3)v^2-2(1 + u^2)(u^2-2u-1)v^4,\\
X_{2}&=25n^2(u^2+2u-1) + 10 n (1 - 2 u + 3 u^2 + 2 u^3)v^2+2 (1 + u^2)(u^2+2u-1)v^4,\\
X_{3}&=25 n^2(1-2 u-u^2)-10n(1 + 3 u^2)v^2+2(1 + u^2)(u^2-2 u-1)v^4,\\
X_{4}&=25n^2(u^2-2u-1) - 10 n (1 + 3 u^2)v^2-2(1 + u^2) (u^2+2u-1)v^4,
\end{align*}
 with $T$ given by
 \begin{equation*}
 T=10n(u^2-1)v(25 n^2 (1 + u^2)+10nu(3 + u^2)v^2+2(1 + u^2)^2v^4)(X_{1}+X_{2}).
\end{equation*}
%\begin{equation*} % T=40nv(u^2-1)(5n+2uv^2)(5nu+(1+u^2)v^2)(25 n^2 (1 + u^2)+10nu(3 + u^2)v^2+2(1 + u^2)^2v^4). %\end{equation*}
It is clear that if $n\equiv 1\pmod{2}$ then the polynomials $X_{1}, X_{2}, X_{3}, X_{4}$ have no common constant factor. Indeed, in this case the $\gcd$ of the coefficients of the polynomial $X_{i}$ for $i=1, 2, 3, 4$ is 1. However, in case $n\equiv 0\pmod{2}$, they have
constant common factor $d=2$. Fortunately, we can easily adjust our solution in order to get co-prime polynomials in this case too. Indeed, if $n=2^{\alpha}n_{1}$ with $n_{1}\equiv 1\pmod{2}$, we replace $u$ by
$2u$ and $v$ by $2^{\beta}v$, where $\beta= \lfloor\frac{\alpha}{2}\rfloor$. After this substitution the polynomial $X_{i}$ for $i=1,2,3,4$, is divisible by
$2^{2\alpha}$, and thus we can eliminate this common factor. Then we get that the coefficient of the polynomial $X_{i}$ free of $v$ is odd for all $u\in\Z$ and
$i=1,2,3,4$. We thus get the statement in case of $n\in\Z\setminus\{0\}$.

In order to get the result for a rational $n$ we write $n=a/b$ with $a, b\in\Z\setminus\{0\}$ and $\gcd(a,b)=1$. From the above we know that the equation under
consideration has polynomial solution $X_{1},\ldots, X_{4}$ for $n=ab$. One can easily check that $X_{1}(u,bv)/b^2,\ldots,X_{4}(u,bv)/b^{2}$ are polynomials with
integer coefficients and  they solve the equation (\ref{5powers}) with $n=a/b$. Our result follows.

\end{proof}

As an immediate consequence from the above theorem we get the following:

\begin{cor}
Let $n\in\Q\setminus\{0\}$ be given. Then the Diophantine equation {\rm (\ref{5powers})} has infinitely many solutions in co-prime integers $X_{1}, X_{2}, X_{3}, X_{4}$.
\end{cor}
\begin{proof}
Let $n=a/b$ and put $u=0$. Next, let us put $v=5bt$. After this substitution the expression $Y_{i}(a,b,t):=(5n)^{-2}X_{i}$ is a polynomial in $\Z[a,b,t]$. It is also clear that $Y_{1}, Y_{2}, Y_{3}, Y_{4}$ satisfy the equation (\ref{5powers}). Moreover, the following identity holds:
$$
-100abt^2Y_{1}(a,b,t)+Y_{2}(a,b,t)-(1 + 130 a b t^2)Y_{3}(a,b,t)-3 (1 + 10 a b t^2)Y_{4}(a,b,t)=1
$$
which implies that for each $a, b, t\in\Z$ we have $\gcd(Y_{1}, Y_{2}, Y_{3}, Y_{4})=1$.
\end{proof}

\begin{rem} {\rm It is worth to note that the sequence of polynomials
\begin{equation*}
F_{n}=(x+y+z)^{2n+1}+(-x+z-y)^{2n+1}+(x-y-z)^{2n+1}+(-x+y-z)^{2n+1},\quad n\in\N_{+},
\end{equation*}
possesses an interesting arithmetic property. However, before we state the mentioned property we introduce the notation of a $p$-adic "valuation" of a polynomial $F\in\Z[\overline{X}]$ of degree $d$, where $\overline{X}=(x_{1},\ldots, x_{k})$ and $p$ is a prime number. In order to define this
quantity let us write
\begin{equation*}
F(\overline{X})=\sum_{|\alpha|\leq d}a_{\alpha}\overline{X}^{\alpha},
\end{equation*}
where $\alpha=(\alpha_{1},\ldots,\alpha_{k})\in\N ^{k}, |\alpha|=\sum_{i=1}^{k}\alpha_{i}$ and $\overline{X}^{\alpha}=x_{1}^{\alpha_{1}}\cdot\ldots\cdot
x_{k}^{\alpha_{k}}$. Then we define $p$-adic valuation of the polynomial $F$ as $\phi_{p}(F)$, where
\begin{equation*}
\phi_{p}(F)=\op{max}\{k\in\N:\;p^{k}|a_{\alpha}\;\mbox{for all}\;\alpha\;\mbox{for which}\;a_{\alpha}\neq 0 \},
\end{equation*}
i.e. the non-negative number
$\phi_{p}(F)$ is just the highest power of $p$ such that the polynomial $F/p^{\phi_{p}(F)}$ has integer coefficients. We also define $\phi_{p}(0)=\infty$. Then,
one can also write an useful identity
\begin{equation*}
\phi_{p}(F)=\op{min}\{\nu_{p}(a_{\alpha})\},
\end{equation*} where the minimum is taken over all
coefficients of the polynomial $F$. It is clear that $\phi_{p}(n)=\nu_{p}(n)$ for $n\in\Z$, where $\nu_{p}(n)$ is the usual $p$-adic valuation of an integer $n$.
In the sequel we will need some properties of $\nu_{p}(n)$. Let us recall that for $n\in\N$ we have
\begin{equation*}
\nu_{p}(n!)=\sum_{i=1}^{\infty}\left\lfloor\frac{n}{p^{i}}\right\rfloor=\frac{n-s_{p}(n)}{p-1},
\end{equation*}
where $s_{p}(n)$ is the sum of digits function,
i.e. if $n=\sum_{i=0}^{k}\alpha_{i}p^{i}$ is the representation of $n$ in base $p$, then $s_{p}(n)=\sum_{i=0}^{k}\alpha_{i}$. This function satisfies the
following recurrence relation: $s_{p}(i)=i$ for $i=0,1,\ldots,p-1$ and
\begin{equation*}
s_{p}(ap+b)=s_{p}(a)+b
\end{equation*}
for $b\in\{0,\ldots,p-1\}$. From
the expression for $\nu_{p}(n!)$ we easily deduce the identity
\begin{equation*}
\nu_{p}(n)=\nu_{p}\left(\frac{n!}{(n-1)!}\right)=\frac{s_{p}(n-1)-s_{p}(n)+1}{p-1}.
\end{equation*}
Because $\nu_{p}(n)\geq 0$ for $n\in\N$ we get an useful
inequality
\begin{equation*}
s_{p}(n-1)+1\geq s_{p}(n)
\end{equation*}
for $n\in\N$.

We are ready to prove the following:

\begin{thm} Let $n\in\N_{+}$ and $p$ be a prime number. Then we have the following identity: \begin{equation*} \phi_{p}(F_{n})=\begin{cases}\begin{array}{lll}
                              \nu_{2}(n)+3, &  & \mbox{if}\;p=2, \\
                              1           , &  & \mbox{if}\;2n+1=p^{m}\;\mbox{for some}\;m\in\N_{+}\;\mbox{and}\;p>2,\\
                              0           , &  & \mbox{otherwise}.
                            \end{array}
\end{cases} \end{equation*} Moreover, $F_{n}(x,y,z)\equiv 0\pmod{xyz}$. \end{thm} \begin{proof} We start with the computation of the coefficients of the
polynomial $F_{n}$. In order to shorten the notation we put $p=y+z, q=-y+z$. We then have
\begin{align*}
F_{n}&=(x+p)^{2n+1}+(-x+q)^{2n+1}+(x-p)^{2n+1}+(-x-q)^{2n+1}\\
     &=\sum_{i=0}^{2n+1}\binom{2n+1}{i}(p^{2n+1-i}+(-1)^{i}q^{2n+1-i}+(-p)^{2n+1-i}+(-1)^{i}(-q)^{2n+1-i})x^{i}\\
     &=\sum_{i=1}^{2n+1}\binom{2n+1}{i}(p^{2n+1-i}+(-1)^{i}q^{2n+1-i}+(-p)^{2n+1-i}+(-1)^{i}(-q)^{2n+1-i})x^{i}\\
     &=2\sum_{i=0}^{n-1}\binom{2n+1}{2i+1}(p^{2(n-i)}-q^{2(n-i)})x^{2i+1}\\
     &=2\sum_{i=0}^{n-1}\binom{2n+1}{2i+1}((y+z)^{2(n-i)}-(z-y)^{2(n-i)})x^{2i+1}\\
     &=2\sum_{i=0}^{n-1}\binom{2n+1}{2i+1}\sum_{j=0}^{2(n-i)}\binom{2(n-i)}{j}(z^{j}-(-1)^{j}z^{j})y^{2(n-i)-j}x^{2i+1}\\
     &=4\sum_{i=0}^{n-1}\sum_{j=0}^{n-i-1}\binom{2n+1}{2i+1}\binom{2(n-i)}{2j+1}x^{2i+1}y^{2(n-i-j-1)+1}z^{2j+1}.
\end{align*}
We observe that $(2i_{1}+1,2(n-i_{1}-j_{1}-1)+1,2j_{1}+1)=(2i_{2}+1,2(n-i_{2}-j_{2}-1)+1,2j_{2}+1)$ if and only if $(i_{1},j_{1})=(i_{2},j_{2})$ and
thus we have an equality
\begin{equation*}
\phi_{p}(F_{n})=\op{min}\{\nu_{p}\left(\binom{2n+1}{2i+1}\binom{2(n-i)}{2j+1}\right):\;i=0,\ldots,n-1;
j=0,\ldots,n-i-1\}+2.
\end{equation*}

In order to get the statement of our theorem we start with $p=2$. First of all let us note that if $i=j=0$ then
$$
\nu_{2}\left(\binom{2n+1}{2i+1}\binom{2(n-i)}{2j+1}\right)=\nu_{2}((2n+1)(2n))=\nu_{2}(n)+1,
$$
and thus $\phi_{p}(F_{n})\leq \nu_{2}(n)+1$ for $n\in\N_{+}$.
Let us note that the opposite inequality is equivalent with the proof that for $n\in\N_{+}$ and each $i\in\{0,\ldots,n-1\},\; j\in\{0,\ldots,n-i-1\}$ we have
\begin{equation}\label{s2}
 s_{2}(i)+s_{2}(j)+s_{2}(n-i-j-1)\geq s_{2}(n-1).
\end{equation} This equivalence follows from the expression of $\nu_{2}\left(\binom{2n+1}{2i+1}\binom{2(n-i)}{2j+1}\right)$ in terms of sum of binary digits
function. Inequality (\ref{s2}) is an obvious consequence of a well-known and simple fact that  $s_2(a)+s_2(b) \geq s_2(a+b)$ for all positive integers $a,b$.
Indeed, we have
\begin{equation*}
 0 \leq \nu_2 \left(\binom{a+b}{a} \right) = \nu_2((a+b)!)-\nu_2(a!)-\nu_2(b!)= s_2(a)+s_2(b)-s_2(a+b).
\end{equation*}

Now let us take $p>2$. If $2n+1 = p^m$, then for $n-i = p^{m-1}$ and $2j+1=p^{m-1}$ we get
\begin{equation*}
 \nu_p\left(\binom{2n+1}{2(n-i)}\binom{2(n-i)}{2j+1}\right)=\nu_{p}\left(\binom{p^m}{2p^{m-1}}\binom{2p^{m-1}}{p^{m-1}}\right) = 1.
\end{equation*}
Indeed, we have
\begin{equation*}
 \binom{p^m}{2p^{m-1}} = \frac{p^m}{2p^{m-1}}\binom{p^m-1}{2p^{m-1}-1} = \frac{p}{2}\binom{p^m-1}{2p^{m-1}-1}
\end{equation*}
and we have the following expansions in base $p$:
$$
p^m-1=\overline{(p-1)\ldots(p-1)},\quad 2p^{m-1}-1 = \overline{1(p-1)(p-1)\ldots(p-1)}.
$$
Therefore from Lucas Theorem we get
\begin{equation*}
 \binom{p^m-1}{2p^{m-1}-1} \equiv \binom{p-1}{1} \binom{p-1}{p-1} \ldots \binom{p-1}{p-1} \equiv p-1 \pmod p.
\end{equation*}
Also from Lucas theorem we get $\binom{2p^{m-1}}{p^{m-1}} \equiv \binom{2}{1} \equiv 2 \pmod p$. We proved that $\phi_{p}(F_{n}) \leq 1$. On the other
hand for all $i < n$ we have $p | \binom{p^m}{2i+1} = \frac{p^m}{2i+1} \binom{p^m-1}{2i}$, therefore $\phi_{p}(F_{n})=1$.

Now let us take $n$ such that $2n+1 \neq p^{m}$. There are two possibilities, either $2n+1 = up^m$ for $3 \leq u < p$, or in $p$ base expansion $2n+1$ has at
least two non-zero digits. In the first case we take $n-i = p^m$ and $2j+1=p^m$ and use Lucas theorem to get
\begin{equation*}
 \binom{2n+1}{2(n-i)} \binom{2(n-i)}{2j+1} = \binom{up^m}{2p^m}\binom{2p^m}{p^m} \equiv \binom{u}{2} \binom{2}{1} \not\equiv 0 \pmod p.
\end{equation*}
In the second case let $2n+1$ has non-zero digits on positions $a$ and $b$, where $a>b$. We take $2(n-i) = p^a+p^b$ and $2j+1=p^b$. Again from
Lucas Theorem we get that $p \nmid \binom{2n+1}{p^a+p^b} \binom{p^a+p^b}{p^b}$. We proved that $\phi_{p}(F_{n})=0$.

Finally, in order to get the second property of $F_{n}$ we observe that $F_{n}(0,y,z)=F_{n}(x,0,z)=F_{n}(x,y,0)=0$ and thus $F_{n}(x,y,z)\equiv 0\pmod{xyz}$.
\end{proof}

Based on small numerical experiments we believe that the following is true:

\begin{conj}
Let $n\in\N_{+}$ and put $\phi(F_{n})=\prod_{p\mid n}p^{\phi_{p}(F_{n})}$, where the product is taken over prime numbers dividing $n$. Then, the
polynomial $$ \frac{F_{n}(x,y,z)}{\phi(F_{n})xyz} $$ is irreducible in $\Z[x,y,z]$.
\end{conj}
}
\end{rem}

\section{Some remarks concerning the equation $T^2=aX_{1}^5+bX_{2}^5+cX_{3}^5+dX_{4}^5$}\label{sec4}

Motivated by Theorem \ref{fifth} obtained in section \ref{sec3} one can ask whether it is possible to prove more general results concerning the Diophantine equation of the form
\begin{equation}\label{quineq}
T^2=aX_{1}^5+bX_{2}^5+cX_{3}^5+dX_{4}^5=:G(\overline{X}), \quad \overline{X}=(X_{1},X_{2},X_{3},X_{4}),
\end{equation}
where $a, b, c, d\in\Z$ are given. It is interesting that according to our best knowledge there is no result concerning the existence of primitive integer solutions of the equation (\ref{quineq}) where at least three among the integers $a, b, c, d$ are non-zero. The aim of this
section is presentation of a method which sometimes will allow us to find infinitely many primitive solutions in integers of the equation (\ref{quineq}). In
order to present our method we assume that there are non-zero integers $x_{1}, x_{2}, x_{3}, x_{4}$ satisfying the Diophantine equation
\begin{equation*}
aX_{1}^5+bX_{2}^5+cX_{3}^5+dX_{4}^5=0.
\end{equation*}

We should note that this condition is quite restrictive due to the fact that we do not even know whether there is an example of integers $a, b, c, d$ such that the above equation has infinitely many solutions in co-prime integers.

We put
\begin{equation}\label{specialsub}
X_{1}=pU+x_{1}V,\quad X_{2}=qU+x_{2}V,\quad X_{3}=rU+x_{3}V, \quad X_{4}=sU+x_{4}V
\end{equation}
and look for the values of $p, q, r, s$ such that the homogenous polynomial $G(\overline{X})$ with $X_{i}, i=1,2,3,4,$ given by (\ref{specialsub}), is divisible by a square of a linear polynomial. In this case we have $G(x,y,z,w)=\sum_{i=1}^{5}C_{i}U^{i}V^{5-i}$, where
\begin{equation*}
C_{i}={5\choose i}(apx_{1}^{5-i}+bqx_{2}^{5-i}+crx_{3}^{5-i}+dsx_{4}^{5-i})\in\Z[p,q,r,s]
\end{equation*}
for $i=1, 2, 3, 4, 5$. Thus if $C_{1}=0$ then $G(x,y,z,w)$ will be divisible by $U^2$. We put
\begin{equation*}
s=-\frac{a p x_1^4+b q x_2^4+crx_{3}^4}{dx_4^4}.
\end{equation*}
After this substitution the polynomial $G$ (treated as a polynomial in variables $U, V$ with coefficients dependent on $x_{0}, y_{0},
z_{0}, w_{0}$ and $p, q$) is divisible by $U^2$. Thus, it can be written as $G(x,y,z)=G(U,V)=U^2H(U,V)$, where $H$ is the homogenous form, in variables $U, V$, of degree 3. The form $H$ is rather complicated and thus we do not present it in the explicit form. However, our reasoning shows that in order to find primitive solutions of the equation (\ref{quineq}) it is enough to find primitive solutions of the equation
\begin{equation}\label{specialequation}
T^2=H(U,V).
\end{equation}
A general expectation is that this type of equations should have infinitely many primitive solutions and in this case it should be possible to find parametric solution. However, the difficulties of finding demanded solutions highly depend on the coefficients $a, b, c, d$ and the chosen solution of $aX_{1}^5+bX_{2}^5+cX_{3}^5+dX_{4}^5=0$. We should note that this type of equations is discussed by Mordell \cite[Chap. 14, p. 112 and Chap. 25]{Mor} (however he mainly deals with the case when the leading coefficients of $H$ in $U$ or $V$ is equal to 1 - the condition which make things slightly simpler).

Let us also note that if we are lucky and the system of equations
$$
C_{1}=C_{2}=0
$$
has a solution for, say, $r, s$, then the polynomial $H$ constructed above will be reducible (and divisible by $U$). Then, it is quite likely that there are integers $p, q$ such that the equation (\ref{specialequation}) has solution in co-prime polynomials  which can be used in order to get co-prime polynomials satisfying (\ref{quineq}). However, in general the constructed polynomial solutions will be co-prime but over polynomial ring with {\it rational} coefficients.

As an application of our method we prove two results. Here is the first one:

\begin{thm}\label{fifth-general}
Let $m, n$ be non-zero integers. Then the Diophantine equation
\begin{equation}\label{5powerster}
T^2=m(X_{1}^5-X_{2}^5)+n^2(X_{3}^5-X_{4}^5)
\end{equation}
has solution in co-prime polynomials $X_{1}, X_{2}, X_{3}, X_{4}, T\in\Z[u]$. In particular, the Diophantine equation {\rm (\ref{5powerster})} has infinitely many primitive solutions in integers.
\end{thm}
\begin{proof}
We follow the presented method with $(x_{1},x_{2},x_{3},x_{4})=(1,1,1,1)$ and take $X_{i}$ as in (\ref{specialsub}). We then have $G(X_{1},X_{2},X_{3},X_{4})=\sum_{i=1}^{5}C_{i}U^{i}V^{5-i}$ and consider the system of equations $C_{1}=C_{2}=0$ which in explicit form is of the shape:
$$
5(mp - mq + n^2r - n^2s)=0,\quad 10(mp^2 - mq^2 + n^2r^2 - n^2s^2)=0.
$$
We solve this system in $r, s$ and get
$$
r=\frac{(n^2-m)p + (n^2+m)q}{2n^2},\quad s=\frac{(n^2+m)p + (n^2-m)q}{2n^2}.
$$
Substituting now the computed values of $r$ and $s$ we note that up to multiplication by square (which can be neglected) we get
\begin{align*}
H(U,V)=&m (n^4 - m^2)(p-q)U\times \\
       &(40n^4V(V+(p+q)U)+((m^2+11n^4)(p^2+q^2)+2(9 n^4-m^2)pq)U^2).
\end{align*}
In the light of our reasoning we need to find integer solutions of the equation $T^2=H(U,V)$. However, in order to do this we introduce new variables $U_{1}. V_{1}$ in the following way:
$$
U=\frac{1}{p-q}U_{1},\quad V=\frac{1}{n^2}V_{1}-\frac{p+q}{2(p-q)}U_{1}.
$$
With $U, V$ given above we get that $H$ takes the form
$$
H(U,V)=H_{1}(U_{1},V_{1})=m(n^4-m^2)U_{1}(40V_{1}^2+(m^2 + n^4)U_{1}^2).
$$
We thus consider the equation $T^2=H_{1}(U_{1},V_{1})$ and in order to find solutions its we are interested in, we put
$$
U_{1}=10m(n^4 - m^2)U_{2}^2, \quad T=10m(n^4 - m^2)U_{2}T_{1}
$$
and left with the equation
$$
T_{1}^2-4V_{1}^2=10m^2(m^2 - n^4)^2(m^2+n^4)U_{2}^4.
$$
This equation can be easily solved by taking  $T_{1}, V_{1}, U_{2}$ as the solutions of the system of equations
$$
T_{1}-2V_{1}=10 m^2(m + n^2)^2(m^2 + n^4), \quad T_{1} + 2V_{1}=(m - n^2)^2t^4,\quad U_{2}=t.
$$
We get $U_{2}=t$ and
\begin{align*}
V_{1}&=\frac{1}{4}(10 m^2(m^2 - n^4)^2(m^2 + n^4)t^4-1)),\\
T_{1}&=\frac{1}{2}(10 m^2(m^2 - n^4)^2(m^2 + n^4)t^4+1).
\end{align*}
Using now $V_{1}$ and $T_{1}$ and performing all necessary simplifications we get the solutions of the equation (\ref{5powerster}) in the following form
\begin{align*}
X_{1}&=10 m^2(m^2 - n^4)^2(m^2 + n^4)t^4-20 m n^2 (m^2 - n^4) t^2-1,\\
X_{2}&=10 m^2(m^2 - n^4)^2(m^2 + n^4)t^4+20 m n^2 (m^2 - n^4) t^2-1,\\
X_{3}&=10 m^2(m^2 - n^4)^2(m^2 + n^4)t^4+20m^2(m^2 - n^4)t^2-1,\\
X_{4}&=10 m^2(m^2 - n^4)^2(m^2 + n^4)t^4-20m^2(m^2 - n^4)t^2-1,
\end{align*}
with $T$ given by
$$
T=400m^2n(n^4 - m^2)^2t^3(10 m^2(m^2 - n^4)^2(m^2 + n^4)t^4+1).
$$
It is clear that in order to have non-trivial solutions we need to assume $m^2-n^4\neq 0$. Indeed, if $m^2=n^4$ then the right hand side of (\ref{5powerster}) with the computed $X_{i}, i=1,2,3,4$ vanish. However, in this case we can invoke the statement of Theorem \ref{fifth} and get non-trivial solutions.
Let us observe that under the assumption $m^2-n^4\neq 0$ we have
$$
\op{gcd} _{\Z[t]}(X_{1}(t), X_{2}(t),X_{3}(t),X_{4}(t))\Big|40m^2(m^2 - n^4)
$$
which is consequence of the identity $X_{1}+X_{2}-2X_{3}=-40m^2(m^2 - n^4)t^4$ and non-vanishing of the value of $X_{i}(0)$ for $i=1, 2, 3, 4$. Thus if, we substitute $t=40m^2(m^2 - n^4)u$ and consider the corresponding polynomials $Y_{i}(u), i=1,2,3,4$ then we will have $Y_{i}(u)\equiv -1\pmod{40m^2(m^2 - n^4)}$. This implies the identity
$$
\op{gcd}_{\Z[u]}(Y_{1}(u), Y_{2}(u),Y_{3}(u),Y_{4}(u))=1.
$$
As a consequence of our construction we see that for any $u\in\Z$ and $X_{i}=Y_{i}(u)$ we have $\op{gcd}(X_{1},X_{2},X_{3},X_{4})=1$ and our result follows.
\end{proof}

It is quite interesting that the method we just presented can be also useful in investigations concerning the existence of primitive integer solutions of the equations
\begin{equation}\label{quineq3}
t^2=ax^5+by^5+cz^5.
\end{equation}
Indeed, if we assume that $d=0$ and $(x_{0}, y_{0}, z_{0})$ is solution of the equation
\begin{equation}\label{Fal}
ax^5+by^5+cz^5=0
\end{equation}
then using our method with
$$
x=pU+x_{0}V,\quad y=qU+y_{0}V,\quad z=rU+z_{0}V
$$
and
$$
r=-\frac{apx_{0}^4+bqy_{0}^4}{cz_{0}^4}
$$
we note that in order to find primitive solutions of the equation (\ref{quineq3}) it is enough to find primitive solutions of the equation of the form
\begin{equation*}
T^2=H(U,V),
\end{equation*}
where $H$ is cubic form depending on the solution  $(x_{0}, y_{0}, z_{0})$ of (\ref{Fal}). We are aware that this condition is even more restrictive then the one related to $d\neq 0$. This is due to the fact that the genus of the associated curve is 6. Thus, according to Fatings theorem there are only finitely many co-prime integer solutions of this equation satisfying the condition $xyz\neq 0$. However, as the next theorem shows, sometimes we can apply this method in order to get solutions of the equation (\ref{quineq3}).

\begin{thm}
Let $n\in\Q\setminus\{0\}$ be given. Then the Diophantine equation
\begin{equation}\label{5powersbis}
T^2=n(X_{1}^5+X_{2}^5-2X_{3}^5)
\end{equation}
has solution in co-prime polynomials $X_{1}, X_{2}, X_{3}, T\in\Z[t,w]$.
\end{thm}
\begin{proof} First we consider the case $n\in\Z\setminus\{0\}$. The equation
$X_{1}^5+X_{2}^5-2X_{3}^5=0$ has solution $(1,1,1)$ and following the method we described above with $p=1, q=3$ we put
\begin{equation*}
X_{1}=U+V,\quad X_{2}=3U+V,\quad X_{3}=2U+V.
\end{equation*}
For $X_{1}, X_{2}, X_{3}$ defined in this way we have
\begin{equation*}
n(X_{1}^5+X_{2}^5-2X_{3}^5)=10nU^2(2U+V)(9U^2+8UV+2V^2).
\end{equation*}
In order to find
solutions we consider the system
\begin{equation}\label{sys}
10n(2U+V)=\square,\quad 9U^2+8UV+2V^2=\square.
\end{equation}
The solution of the second equation
can be easily found with the standard method of projection form rational point with $U=1, V=0$, and we get parametrization
\begin{equation*}
 U=-u^2 + 2 v^2, \quad V=2(3u-4v)v.
 \end{equation*} For $U, V$ given above we have $10n(2U+V)=-20n(u-2v)(u-v)$ and in order to make this expression a square we solve the system
$-5n(u-2v)=(5nt)^2, 4(u - v)=(2w)^2$ with respect to $u, v$ and get
\begin{equation*}
u=5nt^2+2w^2,\quad v=5nt^2+w^2.
\end{equation*}
Finally, tracing back your
reasoning, we get the solution of the equation (\ref{5powersbis}) in the following form:
\begin{equation*}
X_{1}=2w^4+10nt^2w^2-25n^2t^4,\quad X_{2}=25n^2t^4+10nt^2w^2-2w^4, \quad X_{3}=10nt^2w^2.
\end{equation*} The corresponding value of $T$ takes the form $T=5ntw(25n^2t^4+2w^4)(X_2-X_1)$.  If
$n\equiv 0\pmod{2}$ then our solution has constant common factor equal to 2 and need to be slightly different. In order to eliminate it we put
$n=2^{\alpha}n_{1}$ with $n_{1}\equiv 1\pmod{2}$. Then we replace $w$ by $2^{\beta}w$, where $\beta=\lfloor\frac{\alpha}{2}\rfloor$. After this substitution the
polynomial $X_{i}$ for $i = 1, 2, 3$, is divisible by $2^{2\alpha}$, and thus we can eliminate this common factor. Then we get that the coefficient of the
polynomial $X_{i}$ free of $t$ is odd for all $w\in\Z$ and $i=1, 2, 3$. We thus get the statement in case of $n\in \Z \setminus \{0\}$.

In order to get solutions in case of $n=a/b\in\Q\setminus\Z$, with $\gcd(a,b)=1$, we perform exactly the same reasoning as at the end of the proof of Theorem
\ref{fifth}, i.e., we replace $n$ by $ab$ and $w$ by $bw$. Our result follows.
\end{proof}

\begin{rem}\label{Bremner0} {\rm Of course our method has limitations and can not be used in all situations. For example, it can not be used in order to tackle the Diophantine equation
\begin{equation}\label{Bremner1}
T^2=X_{1}^5+X_{2}^5+X_{3}^5
\end{equation}
which has strong classical flavour. Indeed, one can say: {\it Find three fifth powers of non-zero integers without sub-sum equal to zero, which sum to square}. Remarkably, it is possible to find infinitely many solutions of the above equation using "experimental" approach. We owe the following reasoning for Andrew Bremner which showed how one can deal with equation (\ref{Bremner1}).

First of all the list of the primitive solutions of (\ref{Bremner1}) was prepared with $X_{1}, X_{2}, X_{3}$ satisfying the conditions
$$
|X_{1}|+|X_{2}|+|X_{3}|\leq 7500 \quad \mbox{with}\quad X_{1}\leq X_{2}\leq X_{3}.
$$
A clever investigations of the list reveals that there are many solutions for which $X_{1}=a^2$ and $X_{1}+X_{2}+X_{3}=b^2$ for some $a, b\in\Z$. Looking at these solutions, it was noticed that $X_{1}^5+X_{2}^5+X_{3}^5$ was the square of $b^5-m(x+z)(y+z)$ for small integer $m$. Further investigation reveals that $m$ should be taken as $m=5(b-a)$. Thus:
$$
X_{2}=-a^2 + b^2 - X_{1},\quad X_{3}=a^2,\quad T=b^5-5(b-a)(b^2 - x) (a^2 + x).
$$
With $X_{2}, X_{3}$ and $T$ chosen in this way the following identity comes to light
\begin{equation}\label{Bremner2}
T^2-(X_{1}^5+X_{2}^5+X_{3}^5)=5(b-a)(b^2-X_{1})(a^2+X_{1})G(X_{1}),
\end{equation}
where
$$
G(X_{1})=2(3 a - 2 b)X_{1}^2+2(3 a - 2 b)(a^2 - b^2)X_{1}+(a - b)^2 (a^3 + 3 a^2 b - a b^2 - b^3).
$$
Thus, the right hand side of (\ref{Bremner2}) vanish if and only if $(b-a)(b^2-X_{1})(a^2+X_{1})G(X_{1})=0$. The first three factors are not of interest. If $a, b\in\Z$ then last factor vanish for certain $X_{1}\in\Z$ if and only if the discriminant of $G$ is a square in $\Z$, i.e.
$$
\Delta(G)=a(3a-2b)(a-b)^4=\square.
$$
One can take $a=-2t^2, b=1-3t^2$ and this choice give the solution of (\ref{Bremner1}) in the following form
\begin{equation*}
X_{1}=\frac{1}{2}(t^2-1)(t^3+5t^2-t-1),\quad X_{2}=\frac{1}{2}(1-t^2)(t^3-5t^2-t+1),\quad X_{3}=4t^4.
\end{equation*}
Taking $Y_{i}=\frac{1}{4}X_{i}(2u+1)\in\Z[u]$ and noting the identity
$$
8Y_{2}(u)+(1 - 18 u - 44 u^2 - 8 u^3 + 16 u^4)(2u+1)=1
$$
we get $\op{gcd}_{\Z[u]}(Y_{2}(u),Y_{3}(u))=1$ and thus $\op{gcd}_{\Z[u]}(Y_{1}(u),Y_{2}(u),Y_{3}(u))=1$. Summing up, the following is true:

\begin{thm}\label{Bremner}
The Diophantine equation $T^2=X_{1}^5+X_{2}^5+X_{3}^5$ has infinitely many solutions in integers satisfying the condition $X_{i}^2\neq X_{j}^2$ for $i\neq j$ with $i, j\in\{1,2,3\}$.
\end{thm}

}
\end{rem}

\begin{rem}
{\rm
Motivated by the results of this and previous sections, we performed small numerical search for primitive integer solutions of the Diophantine equation (\ref{quineq}) with quadruplets of positive integers $a, b, c$ and $d\in\N$
satisfying the condition $0<a\leq b\leq c\leq 10$ and $0\leq d\leq 10$. Remarkably, in each case we found at least one primitive solution of (\ref{quineq}) among the integers satisfying the condition $\op{max}\{X_{1},X_{2},X_{3},X_{4}\}\leq 100$ in case $d\neq 0$. The same is true in case $d=0$. We thus dare to state the following:

\begin{conj}
\begin{enumerate} \item For each $a, b, c\in\Z\setminus\{0\}$ and $d\in\Z$ the Diophantine equation $T^2=aX_{1}^5+bX_{2}^5+cX_{3}^5+dX_{4}^5$ has at least one solution in co-prime integers $X_{1}, X_{2}, X_{3}, X_{4}$.

\item For each $a, b, c\in\Z\setminus\{0\}$ and $d\in\Z$ the Diophantine equation $T^2=aX_{1}^5+bX_{2}^5+cX_{3}^5+dX_{4}^5$ has infinitely many solutions in co-prime integers $X_{1}, X_{2}, X_{3}, X_{4}$.
\end{enumerate}
\end{conj}

} \end{rem}

\bigskip

\noindent {\bf Acknowledgments}
{\rm We thank Professor Andrew Bremner for kind permission to include his observation concerning the Diophantine equation $T^2=X_{1}^5+X_{2}^5+X_{3}^5$ as a Remark \ref{Bremner0} (and thus Theorem \ref{Bremner}).}

\bigskip
\noindent
Maciej Gawron, Maciej Ulas, Jagiellonian University, Faculty of Mathematics and Computer Science, Institute of Mathematics, {\L}ojasiewicza 6, 30 - 348 Krak\'{o}w, Poland
\\ e-mail:\;{\tt maciekggawron@gmail.com}
\\ e-mail:\;{\tt maciej.ulas@uj.edu.pl}

 \end{document}